\documentclass[12pt]{amsart}
\usepackage{float} 
\usepackage{amsmath,amsthm,amssymb,enumitem,verbatim,stmaryrd,xcolor,microtype,graphicx,aliascnt,needspace,mathtools}
\usepackage[T1]{fontenc}
\usepackage[utf8]{inputenc}
\usepackage[english]{babel} % allow for hyphenation for words with dash: e.g. ``non-archimedean'' does't break, but ``non\hyph archimedean'' does
\usepackage[top=3.5cm,bottom=3.45cm,left=3.5cm,right=3.5cm]{geometry}
\usepackage[bookmarksopen=true,bookmarksopenlevel=2,bookmarksdepth=2,linktoc=page,colorlinks,linkcolor={red!70!black},citecolor={red!80!black},urlcolor={blue!80!black},pdfpagemode=UseOutlines,pdfstartview={fitH}]{hyperref}
\usepackage{longtable}

\usepackage{tikz}
\usepackage{tikz-cd}
\pgfrealjobname{points-for-semirings}% the jobname has to coincide with the file name without the extension .tex in order to precompile images from tikz. To pre-compile images, change to the directory of this file, create a subfolder ./tikz and enter
\usepackage[mathcal]{euscript}                               % use the euscript caligraphic font with \mathscr{...}

\usepackage[colorinlistoftodos,bordercolor=orange,backgroundcolor=orange!20,linecolor=orange,textsize=footnotesize]{todonotes} \makeatletter \providecommand \@dotsep{5} \def\listtodoname{List of Todos} \def\listoftodos{\@starttoc{tdo}\listtodoname} \makeatother %\Todo{} for margin notes, suppress in pdf with option [disable]

\usepackage{newtxtext,newtxmath}                              % times font
\usepackage{bm}

\allowdisplaybreaks % allows equation environments to break between the pages

\usepackage{etoolbox}
\makeatletter
\patchcmd{\@startsection}{\@afterindenttrue}{\@afterindentfalse}{}{}             %omit indentation of the first paragraph of a section
\patchcmd{\part}{\bfseries}{\bfseries\LARGE}{}{}
\patchcmd{\section}{\scshape}{\bfseries}{}{}\renewcommand{\@secnumfont}{\bfseries} %boldface no smallcaps section and subsection titles with numbers
\patchcmd{\@settitle}{\uppercasenonmath\@title}{\large}{}{}
\patchcmd{\@setauthors}{\MakeUppercase}{}{}{}
\addto{\captionsenglish}{} %boldface no smallcaps Abstract
\addto{\captionsenglish}{} %boldface no smallcaps Figure
\addto{\captionsenglish}{} %boldface no smallcaps Table
\makeatother

\usepackage{fancyhdr}

\addto\extrasenglish{}  % Use capitalized ``Section'' (also for (sub)subsections
\theoremstyle{plain}
\newtheorem{thm}{Theorem}[section] % provides command \autoref{}, which produces citations like ``Theorem 1.1''.
\newtheorem{thmx}{Theorem} % provides command \autoref{}, which produces citations like ``Theorem 1.1''.
\newaliascnt{lemma}{thm}\newtheorem{lemma}[lemma]{Lemma}\aliascntresetthe{lemma}
\newaliascnt{cor}{thm}\newtheorem{cor}[cor]{Corollary}\aliascntresetthe{cor}
\newaliascnt{prop}{thm}\aliascntresetthe{prop}
\newaliascnt{claim}{thm}\newtheorem{claim}[claim]{Claim}\aliascntresetthe{claim}
\newtheorem*{claim*}{Claim}
\newtheorem*{thm*}{Theorem}
\newtheorem*{lem*}{Lemma}
\newtheorem*{cor*}{Corollary}

\theoremstyle{definition}
\newaliascnt{df}{thm}\aliascntresetthe{df}
\newaliascnt{rem}{thm}\aliascntresetthe{rem}
\newaliascnt{ex}{thm}\aliascntresetthe{ex}
\newaliascnt{conj}{thm}\aliascntresetthe{conj}
\newaliascnt{problem}{thm}\aliascntresetthe{problem}

\newtheorem*{df*}{Definition}
\newtheorem*{ex*}{Example}
\newtheorem*{rem*}{Remark}

\theoremstyle{remark}

\DeclareRobustCommand{\gobblefour}[5]{}    % Command \SkipTocEntry for suppressing a section title in TOC
\DeclareMathOperator{\Aff}{Aff}

\DeclareMathOperator{\Sch}{Sch}
\DeclareMathOperator{\Sp}{Sp}
\DeclareMathOperator{\Spec}{Spec}
\DeclareMathOperator{\PSpec}{PSpec}
\DeclareMathOperator{\MSpec}{MSpec}
\DeclareMathOperator{\Sh}{Sh}
\DeclareMathOperator{\PrSh}{PrSh}

\DeclareMathOperator{\colim}{colim\,}

\DeclareMathOperator{\SRings}{{SRings}}

\DeclareMathOperator{\Sym}{{Sym}}

\newcommand{\N}{{\mathbb N}}

\newcommand{\cC}{{\mathcal C}}

\newcommand{\cI}{{\mathcal I}}
\newcommand{\cJ}{{\mathcal J}}

\newcommand{\cO}{{\mathcal O}}
\newcommand{\cP}{{\mathcal P}}
\newcommand{\cQ}{{\mathcal Q}}

\newcommand{\cS}{{\mathcal S}}
\newcommand{\cT}{{\mathcal T}}
\newcommand{\cU}{{\mathcal U}}
\newcommand{\cV}{{\mathcal V}}
\newcommand{\cW}{{\mathcal W}}

\newcommand{\fp}{{\mathfrak p}}

\newcommand{\fs}{{\mathfrak s}}

\newcommand{\0}{{\mathbf{0}}}
\newcommand{\1}{{\mathbf{1}}}

\newcommand{\id}{\textup{id}}

\renewcommand{\top}{\textup{top}}

\newcommand{\op}{\textup{op}}

\renewcommand{\leq}{\leqslant}
\newcommand{\gen}[1]{\langle #1 \rangle}

\renewcommand{\emptyset}\varnothing

\title{Two comparison theorems for semiring schemes}

\author{Oliver Lorscheid}
\address{\rm Oliver Lorscheid, University of Groningen, the Netherlands}
\email{o.lorscheid@rug.nl}

\hypersetup{pdftitle={Two comparison theorems for semiring schemes},pdfauthor={Oliver Lorscheid}}

\thanks{The author thanks Alejandro Mart\'inez M\'endez for his comments on an early draft of this paper.}

\begin{document}

\begin{abstract}
 In this note, we compare the two approaches to semiring schemes as topological spaces with a structure sheaf (\cite{GKMX}) and as a functor of points (\cite{BLMV}). We explain and prove the following two results: (1) the topological space can be recovered from the functor of points; (2) the two notions of semiring schemes are canonically equivalent as categories.
\end{abstract}

\maketitle

%%%%%%%%%%%%%%%%%%%%%%%%%%%%%%%%%%%%%%%%%%%%%%%%%%%%%%%%%%%%%%%%%%%%%%%%%%%%%%%%%%%%%%%%%%%%%%%%%%%%%%%%%%%%%%%%%%%%%%%%%%%%%%%%%%%%%%%%%%%%%%%%%%%%%%%%%%%%%%%%%%%%%
%%%%%%%%%%%%%%%%%%%%%%%%%%%%%%%%%%%%%%%%%%%%%%%%%%%%%%%%%%%%%%%%%%%%%%%%%%%%%%%%%%%%%%%%%%%%%%%%%%%%%%%%%%%%%%%%%%%%%%%%%%%%%%%%%%%%%%%%%%%%%%%%%%%%%%%%%%%%%%%%%%%%%

\section*{Introduction}
\label{introduction}

Let $\SRings$ be the category of commutative semirings $R$ with $0$ and $1$. The earlier chapters \cite{BLMV} and \cite{GKMX} of this volume define the spectrum $\Spec R$ in two different ways: once as a functor of points, which is a Zariski sheaf on the opposite category $\Aff=\SRings^\op$ of $\SRings$; and once as the prime spectrum, which is a topological space $X=\PSpec R$ together with a structure sheaf $\cO_X$ of semirings. This yields two approaches to semiring schemes.

In this text, we explain the proofs of two comparison theorems for semiring schemes, which can be stated as follows (for a concise formulation, see \autoref{thmA} and \autoref{thmB} below).

\subsection*{First comparison theorem}
\textit{The prime spectrum $\PSpec R$ of a semiring $R$ in the sense of \cite{GKMX} is canonically identified with the Stone dual $\Lambda_R^\top$ of the locale $\Lambda_R$ of all open subschemes of $\Spec R$ in the sense of \cite{BLMV}.}

\subsection*{Second comparison theorem}
\textit{There is a unique equivalence between the categories of semiring schemes in the senses of \cite{BLMV} and \cite{GKMX} that preserves affine open coverings.}

\medskip

The first comparison result follows in essence from Marty's characterization (in \cite{Marty07}) of points of certain types of relative schemes in the sense of To\"en-Vaqui\'e (\cite{Toen-Vaquie09}), which applies, in particular, to semiring schemes. However, this proof takes rather the scenic route for our purposes: it requires to identify To\"en-Vaqui\'e's relative schemes with certain instances of $\fs$-schemes in the sense of \cite{BLMV} and relies on the comparatively complicated proof of Marty.

In this note, we give a short independent proof of the first comparison theorem, which is tailored to semiring schemes and which builds up directly on the concepts of \cite{BLMV} and \cite{GKMX}. The second comparison theorem is folklore, without a rigorous proof in the literature.

\section{The two comparison theorems}
\label{section: the two comparison theorems}

In the following, we review the approaches to semiring schemes from \cite{BLMV} and \cite{GKMX}, and we explain a concise formulation of the comparison theorems.

\subsection{Prime spectra}

We begin with a description of the prime spectrum of a semiring $R$. A \emph{prime ideal of $R$} is an $R$-submodule $\fp$ of $R$ for which $S=R-\fp$ is a \emph{multiplicative subset}, i.e., $S$ is closed under multiplication and $1\in S$.

The \emph{prime spectrum} $X=\PSpec R$ of $R$ is the set of all prime ideals of $R$ together with the topology that is generated by the \emph{principal open subsets}
\[
 U^P_h \ = \ \{\fp\in X\mid h\notin\fp\}
\]
for $h\in R$. It comes equipped with a sheaf $\cO_X$ of semirings (its \emph{structure sheaf}), which is determined by its values $\cO_X(U^P_h)=B[h^{-1}]$ on principal opens and by the corresponding restriction maps, which are the localizations $B[h^{-1}]\to B[g^{-1}]$ whenever $g=ah$ for some $a\in R$.

In particular, taking global sections recovers $R\simeq\cO_X(X)$ (in a functorial way), which shows that $\PSpec$ identifies $\SRings^\op$ with a subcategory $\cC$ of the category of \emph{semiringed spaces}, i.e., topological spaces equipped with a sheaf of semirings.

\subsection{The Zariski site}
\label{subsection: Zariski site}

The spectrum functor in \cite{BLMV} is the tautological anti-equivalence $\Spec:\SRings\to\Aff$ of the category of semirings with its opposite category $\Aff=\SRings^\op$. The Yoneda embedding $\Aff\to\PrSh(\SRings)$ allows us to consider $\Spec R$ as the functor of points, which is a presheaf on $\SRings$, in analogy to usual algebraic geometry.

Let $R$ be a semiring and $X=\Spec R$. The localization $R\to R[h^{-1}]$ of $R$ at a single element $h\in R$ gives rise to the notion of a \emph{principal open} $U_h=\Spec R[h^{-1}]\to X$ in $\Aff$. The \emph{Zariski topology on $\Aff$} is the Grothendieck pretopology $\cT$ on $\Aff$ that consists of all covering families $\cU=\{\Spec U_{h_i}\to X\mid i\in I\}$ of $X$ by principal opens that appear in the canonical topology of $\Aff$. This defines the \emph{Zariski site} $\fs=(\Aff,\cT)$ and an embedding of $\Aff$ into the topos $\Sh_\fs=\Sh(\Aff,\cT)$. In other words, the functor of points $\Spec R$ is a sheaf in the Zariski topology.

\subsection{The first comparison theorem}

The first comparison theorem recovers the prime spectrum from the properties of $X=\Spec R$ as an object of $\fs$.

Let $\cS=\{U_i\to X\}_{i\in I}$ be a family of principal opens $U_i=\Spec R[h_i^{-1}]$ of $X=\Spec R$, which we do not assume to be in $\cT$. Let $U_{i,j}=U_i\times_X U_j$ be the fiber product. The \emph{atlas of $\cS$} is the diagram $\cU_\cS$ that consists all $U_{i}$ and $U_{i,j}$, together with the canonical projections $U_{i,j}\to U_i$ and $U_{i,j}\to U_j$. Then $\cU_\cS$ is an $\fs$-presentation and $W=\colim_{\Sh_\fs}\cU_\cS$ is an $\fs$-scheme in the sense of \cite{BLMV}; also cf.\ \autoref{subsubsection: s-schemes}.

An \emph{open immersion into $X$} is the canonical morphism $W\to X$ from the colimit $W=\colim_{\Sh_\fs}\cU_\cS$ into $X$, where $\cU_\cS$ is the atlas of a family $\cS$ of principal opens. An \emph{open subscheme of $X$} is the isomorphism class $[W\to X]$ of an open immersion $W\to X$. The \emph{locale of opens of $\Spec R$} is the set of all open subschemes $[W\to X]$, which is partially ordered by the rule that $[W\to X]\leq [W'\to X]$ if and only if there exists a morphism $W\to W'$ over $X$. This poset is a locale (or complete lattice), i.e., every subset $\cS$ of $\Lambda_R$ has a unique least upper bound $\bigvee\cS$ and a unique greatest lower bound $\bigwedge\cS$.

Stone duality associates with $\Lambda_R$ the topological space $\Lambda_R^\top$ of all principal prime ideals of $\Lambda_R$, which are proper subsets of the form $\cP=\{w\in\Lambda_R\mid w\leq \bigvee\cP\}$ for which $\cS=\Lambda_R-\cP$ is closed under finite meets. The topology of $\Lambda_R^\top$ is generated by open subsets of the form $U^\Lambda_h=\{\cP\in\Lambda_R^\top\mid u_h\notin\cP\}$, where $u_h=[U_h\to X]$.

\begin{thmx}[first comparison theorem]\label{thmA}
 There is a canonical homeomorphism $\varphi:\PSpec R\to\Lambda_R^\top$ that is functorial in $R$ and that maps a principal open $U^P_h$ of $\PSpec R$ (for $h\in R$) to the principal open $U^\Lambda_h$ of $\Lambda_R^\top$.
\end{thmx}

\subsection{Semiring schemes}
\label{subsection: semiring schemes}

The earlier chapters \cite{BLMV} and \cite{GKMX} provide two different definitions of semiring schemes based on the previous constructions of the spectrum of a semiring. Our second result shows that the corresponding categories of semiring schemes are canonically equivalent.

\subsubsection{Semiring schemes as semiringed spaces (\cite[section 3.4]{GKMX})}
\label{subsubsection: semiring schemes as semiringed spaces}

A \emph{semiringed space} is a topological space $X$ together with a sheaf $\cO_X$ in $\SRings$. A \emph{local morphism} between semiringed spaces is a continuous map $\varphi:X\to Y$ together with a morphism $\varphi^\#:\varphi^\ast\cO_Y\to\cO_X$ of sheaves that is \emph{local} in the sense that the induced maps $\varphi_x:\cO_{Y,\varphi(x)}\to\cO_{X,x}$ between stalks (for $x\in X$) send non-units to non-units. This defines the category $\Sp_\N$ of semiringed spaces.

An \emph{affine semiring scheme} is a semiringed space that is isomorphic to the prime spectrum $\PSpec R$ of a semiring $R$, as defined above. More precisely, $\PSpec$ defines an anti-equivalence of $\SRings$ with the full subcategory of affine semiring schemes in $\Sp_\N$.

A \emph{semiring scheme}, or \emph{semischeme} in the terminology of \cite{GKMX}, is a semiringed space that has an open covering by affine semiring schemes. We define $\Sch_\N$ as the full subcategory of the category of $\Sp_\N$ whose objects are semiring schemes.

An \emph{open immersion of semiring schemes} is a morphism $\varphi:Y\to X$ of semiring schemes that is an open topological embedding between the underlying topological spaces and for which $\cO_Y$ is the restriction of $\cO_X$ to $Y$, i.e., $\varphi^\#:\varphi^\ast\cO_X\to\cO_Y$ is an isomorphism.

\subsubsection{Semiring schemes as Zariski sheaves (\cite[section 3.2]{BLMV})}
\label{subsubsection: s-schemes}

In order to distinct this approach from the previous one, we use the term $\fs$-scheme in the following, where $\fs=(\Aff,\cT)$ is the Zariski site for $\Aff=\SRings^\op$.

An \emph{affine $\fs$-scheme} is a sheaf in the essential image of the Yoneda embedding $\Aff\to\Sh_\fs$.

A \emph{diagram in $\Aff$} is a functor $\cU:\cI\to\Aff$ from an index category $\cI$ to $\Aff$. We call the images $U_i=\cU(i)$ of objects $i$ of $\cI$ the \emph{objects in $\cU$} and we call the images $\cU(i\to j):U_i\to U_j$ of morphisms $i\to j$ in $\cI$ the \emph{morphisms in $\cU$}. A diagram $\cU$ is \emph{closed under fibre products} if for every pair of morphisms $U_i\to U_k$ and $U_j\to U_k$, the fibre product $U_{ij}=U_i\times_{U_k}U_j$ as well as the two canonical projections are in $\cU$. A diagram $\cU$ is \emph{monodromy free} if for every object $U_i$ in $U$ the canonical inclusion $\iota_i:U_i\to \colim_{\Sh_\fs}\cU$ is a monomorphism in $\Sh_\fs$.

An \emph{$\fs$-presentation} is a monodromy free diagram of principal open immersions in $\Aff$ that is closed under fibre products. An \emph{$\fs$-scheme} is a sheaf $X\in\Sh_\fs$ that is isomorphic to the colimit of an $\fs$-presentation.

\subsection{The second comparison theorem}

The definition of $\fs$-schemes as colimits of $\fs$-presentations allows for a short concise formulation of the comparison theorem.

Let $\Aff_\N$ be the full subcategory of affine semiring schemes in $\Sch_\N$. The composition of the anti-equivalences $\Gamma:\Aff\to\SRings$ and $\PSpec:\SRings\to \Aff_\N$ yields an equivalence of categories $\tau:\Aff\to\Aff_\N$. We write $\tau\cU:\cI\to\Aff_\N$ for the composition $\tau\circ\cU$ of $\tau$ with an $\fs$-presentation $\cU:\cI\to\Aff$.

\begin{thmx}[second comparison theorem]\label{thmB}
 There is an essentially unique extension of $\tau$ to an equivalence $\hat\tau:\Sch_\fs\to\Sch_\N$ that preserves the colimits of $\fs$-presentations, i.e., the canonical morphism $\colim_{\Sch_\N}\tau\cU\to\hat\tau(\colim_{\Sh_\fs}\cU)$ is an isomorphism for every $\fs$-presentation $\cU$ and the canonical inclusions $\iota_i:\tau\cU(i)\to \colim_{\Sch_\N}\tau\cU$ are open immersions.
\end{thmx}

In the rest of the paper, we explain a complete proof for \autoref{thmA}. We explain the proof of \autoref{thmB}, but we omit the more technical aspects of the proof. A complete proof in larger generality is contained in \cite{toolkit}.

%%%%%%%%%%%%%%%%%%%%%%%%%%%%%%%%%%%%%%%%%%%%%%%%%%%%%%%%%%%%%%%%%%%%%%%%%%%%%%%%%%%%%%%%%%%%%%%%%%%%%%%%%%%%%%%%%%%%%%%%%%%%%%%%%%%%%%%%%%%%%%%%%%%%%%%%%%%%%%%%%%%%%
%%%%%%%%%%%%%%%%%%%%%%%%%%%%%%%%%%%%%%%%%%%%%%%%%%%%%%%%%%%%%%%%%%%%%%%%%%%%%%%%%%%%%%%%%%%%%%%%%%%%%%%%%%%%%%%%%%%%%%%%%%%%%%%%%%%%%%%%%%%%%%%%%%%%%%%%%%%%%%%%%%%%%

\section{Proof of \autoref{thmA}}
\label{Proof}

We prove \autoref{thmA} by embedding both spaces $\PSpec R$ and $\Lambda_R^\top$ as induced subspaces into a common larger space and compare their respective images. This common larger space is the monoid spectrum $\MSpec R$ of $R$, as introduced below.

%%%%%%%%%%%%%%%%%%%%%%%%%%%%%%%%%%%%%%%%%%%%%%%%%%%%%%%%%%%%%%%%%%%%%%%%%%%%%%%%%%%%%%%%%%%%%%%%%%%%%%%%%%%%%%%%%%%%%%%%%%%%%%%%%%%%%%%%%%%%%%%%%%%%%%%%%%%%%%%%%%%%%
\subsection{The monoid spectrum}
\label{subsection: the monoid spectrum}

An \emph{prime $m$-ideal} of $R$ is a subset $P$ of $R$ that contains $0$, that is closed under multiplication with elements of $R$ and for which $S=R-P$ is a multiplicative subset. The \emph{monoid spectrum of $R$} is the set $\MSpec R$ of all prime $m$-ideals together with the topology generated by the open subsets
\[
 U^m_h \ = \ \{P\in\MSpec R\mid h\notin P\}.
\]

Evidently every prime ideal of $R$ is a prime $m$-ideal, which defines an inclusion $\PSpec R\hookrightarrow\MSpec R$. Further it is clear that $U^P_h=U^m_h\cap\PSpec R$, which shows that this inclusion is a topological embedding.

Similarly there is a canonical embedding of $\Lambda_R^\top$ into $\MSpec R$ in terms of pullbacks along the \emph{characteristic map $\chi_R:R\to\Lambda_R$}, which maps an element $h\in R$ to the isomorphism class $u_h$ of the principal open $U_h\to X$.

\begin{lemma}\label{lemma: embedding of the Stone dual into MSpec R}
 The inverse image $\chi_R^{-1}(\cP)$ of a principal prime ideal $\cP$ of $\Lambda_R$ is a prime $m$-ideal of $R$. The corresponding map
 \[
  \chi_R^\ast: \ \Lambda_R^\top \ \longrightarrow \ \MSpec R.
 \]
 is a topological embedding.
\end{lemma}

\begin{proof}
 Let $\cP$ be a principal prime ideal of $\Lambda_R$ with its inverse image $P=\chi_R^{-1}(\cP)$ in $R$ and let $S=R-P$. Then $0\in P$ since $R[0^{-1}]=\{0\}$ is the trivial semiring and thus $\chi_R(0)=[\emptyset\to X]=\0$. Further $1\in S$ since $R[1^{-1}]=R$ and thus $\chi_R(1)=[X\to X]=\1$. Since $U_{gh}=U_g\times_XU_h$, we have $\chi_R(g\cdot h)=[U_g\to X]\wedge [U_h\to X]$. Thus $gh\in P$ for $g\in P$ and $h\in R$ and $gh\in S$ for $g,h\in S$, which shows that $P$ is a prime ideal.

 The map $\chi_R^\ast$ is injective since a principal prime ideal $\cP=\{w\in\Lambda_R\mid w\leq\bigvee\cP\}$ is generated by $\bigvee\cP$ and since
 \[\textstyle
  \bigvee\cP \ = \ \bigvee\{u_h\in\Lambda_R\mid h\in R, \ u_h\in\cP\} \ = \ \bigvee\{u_h\in\Lambda_R\mid h\in P\}.
 \]
 by the very definition of open subobjects in terms of colimits of principal opens. Since the topology of $\Lambda^\top_R$ is generated by the open subsets of the form $U^\Lambda_h$ with $h\in R$ and since $U^\Lambda_h=U^m_h\cap\Lambda^\top_R$, this inclusion is a topological embedding.
\end{proof}

%%%%%%%%%%%%%%%%%%%%%%%%%%%%%%%%%%%%%%%%%%%%%%%%%%%%%%%%%%%%%%%%%%%%%%%%%%%%%%%%%%%%%%%%%%%%%%%%%%%%%%%%%%%%%%%%%%%%%%%%%%%%%%%%%%%%%%%%%%%%%%%%%%%%%%%%%%%%%%%%%%%%%
\subsection{Proof strategy}
\label{subsection: proof strategy}

\autoref{lemma: embedding of the Stone dual into MSpec R} allows us to consider both $\PSpec R$ and $\Lambda^\top_R$ as topological subspaces of $\MSpec R$, endowed with the subspace topology. Thus all that is needed to show is that $\PSpec R$ and $\Lambda_R^\top$ coincide as subsets of $\MSpec R$. It is clear from their analogous definitions that the respective principal open subsets $U^P_h$ and $U^\Lambda_h$ agree.

The homeomorphism is functorial in $R$ since $\MSpec R$ is functorial in $R$: a semiring homomorphism $f:R\to S$ defines a continuous map $f^\ast:\MSpec S\to\MSpec R$ by taking the inverse image of prime $m$-ideals.

%%%%%%%%%%%%%%%%%%%%%%%%%%%%%%%%%%%%%%%%%%%%%%%%%%%%%%%%%%%%%%%%%%%%%%%%%%%%%%%%%%%%%%%%%%%%%%%%%%%%%%%%%%%%%%%%%%%%%%%%%%%%%%%%%%%%%%%%%%%%%%%%%%%%%%%%%%%%%%%%%%%%%
\subsection{Technique \#1: localizations}
\label{subsection: localizations}

Let $P$ be a prime $m$-ideal of $R$ and assume we intend to show that $P$ is in $\Lambda_R^\top$ or $\PSpec R$. Then we can assume that $P$ is the unique maximal $m$-ideal of $R$ due to the following fact.

\begin{lemma}\label{lemma: localizations at primes}
 The extension $S^{-1}P=\{\frac as\mid a\in P,s\in S\}$ is the unique maximal $m$-ideal of $S^{-1}R$ and $(S^{-1}R)^\times=S^{-1}R-S^{-1}P$. Moreover, $P=\iota_S^{-1}(S^{-1}P)$, where $\iota_S:R\to S^{-1}R$ is the localization map, and $P$ is in $\PSpec R$ (resp.\ $\Lambda^\top_R$) if $S^{-1}P$ is in $\PSpec S^{-1}R$ (resp.\ $\Lambda^\top_{S^{-1}R}$).
\end{lemma}

\begin{proof}
 The assertions about $S^{-1}P$ (it's the unique maximal $m$-ideal and the complement of the unit group, and that its inverse image in $R$ is $P$) are standard results about the localization of (pointed) monoids; for instance see \cite[section 2.1.2]{Chu-Lorscheid-Santhanam12}. See \cite[Lemma 1.11]{GKMX} for the fact that the inverse image of a prime ideal is a prime ideal.

 Assume that $S^{-1}P=\chi_R^{-1}(\cQ)$ for a principal prime ideal $\cQ$ of $\Lambda_R$. Let $\iota_{S,\ast}:\Lambda_R\to\Lambda_{S^{-1}R}$ be the order-preserving map induced by $\iota_S:R\to S^{-1}R$ and $\cP=\iota^{-1}_{S,\ast}(\cQ)$, which is a principal prime ideal of $\Lambda_R$. Since $\chi_{S^{-1}R}\circ\iota_S=\iota_{S,\ast}\circ\chi_R$, we have $P=\chi_R^{-1}(\cP)$, which shows that $P$ is in $\Lambda^\top_R$.
\end{proof}

This allows us to assume that $P$ is the unique maximal $m$-ideal of $R$ and that $R^\times=R-P$.

%%%%%%%%%%%%%%%%%%%%%%%%%%%%%%%%%%%%%%%%%%%%%%%%%%%%%%%%%%%%%%%%%%%%%%%%%%%%%%%%%%%%%%%%%%%%%%%%%%%%%%%%%%%%%%%%%%%%%%%%%%%%%%%%%%%%%%%%%%%%%%%%%%%%%%%%%%%%%%%%%%%%%
\subsection{Technique \#2: covering conditions}
\label{subsection: covering conditions}

The comparison $\PSpec R$ and $\Lambda_R^\top$ relies on explicit characterizations of the respective covering conditions. For $\PSpec R$, such a condition is described in \cite[Lemma 3.7]{GKMX}:

\begin{lemma}\label{lemma: covering condition for PSpec R}
 Let $\{h_i\mid i\in I\}$ be a family of elements of $R$. Then $\PSpec R=\bigcup_{i\in I}U^P_{h_i}$ if and only if $\gen{h_i\mid i\in I}=R$.
\end{lemma}

For $\Lambda_R^\top$, we have the following covering condition. An $R$-algebra is a semiring homomorphism $\alpha:R\to T$. For $h\in R$, we define $T[h^{-1}]$ as $T[\alpha(h)^{-1}]$.

\begin{lemma}\label{lemma: covering condition for the Stone dual}
 Let $\{h_i\mid i\in I\}$ be a family of elements of $R$. Then $\1=\bigvee\{u_{h_i}\mid i\in I\}$ if and only if for every $R$-algebra $T$, the canonical morphism
 \begin{equation*}
  T \ \longrightarrow \ \lim\bigg(\prod_{i\in I} T[h_i^{-1}] \begin{array}{c}\longrightarrow\\[-7pt]\longrightarrow\end{array} \prod_{i,j\in I} T[(h_ih_j)^{-1}] \bigg)
 \end{equation*}
 is an isomorphism.
\end{lemma}

\begin{proof}
 By the definition of $\Lambda_R$, the smallest upper bound $\bigvee\{u_{h_i}\mid i\in I\}$ of the $u_{h_i}$ is the isomorphism class $w=[W\to X]$ of the colimit
 \[
  W \ = \ \colim\bigg( \coprod_{i,j\in I} \ U_{h_ih_j} \begin{array}{c}\longrightarrow\\[-7pt]\longrightarrow\end{array} \coprod_{i\in I} \ U_{h_i} \bigg)
 \]
 (where the colimit is taken in $\Sh_\fs$). Thus $\bigvee\{u_{h_i}\mid i\in I\}=\1$ if and only if $W=X$. This means that $\{U_{h_i}\to X\}$ is in the canonical topology for $\Aff$ and therefore in $\cT$ (as a family of principal opens). Since covering families of the canonical topology are universally effective-epimorphic, this means (when translated to the respective ``coordinate algebras'' in $\SRings$) that for every semiring homomorphism $R\to T$, the induced homomorphism
 \[
  T \ \longrightarrow \ \lim\bigg(\prod_{i\in I} T[h_i^{-1}] \begin{array}{c}\longrightarrow\\[-7pt]\longrightarrow\end{array} \prod_{i,j\in I} T[(h_ih_j)^{-1}] \bigg)
 \]
 (where the limit is taken in $\SRings$) is an isomorphism.
\end{proof}

%%%%%%%%%%%%%%%%%%%%%%%%%%%%%%%%%%%%%%%%%%%%%%%%%%%%%%%%%%%%%%%%%%%%%%%%%%%%%%%%%%%%%%%%%%%%%%%%%%%%%%%%%%%%%%%%%%%%%%%%%%%%%%%%%%%%%%%%%%%%%%%%%%%%%%%%%%%%%%%%%%%%%
\subsection{Technique \#3: symmetric algebras}
\label{subsection: symmetric algebras}

An $R$-algebra is a semiring homomorphism $R\to T$ and an $R$-module is an additive monoid $M$ together with an action of $R$. The forgetful functor from $R$-algebras to $R$-modules (by forgetting the multiplicative structure of an $R$-algebra) has a left adjoint, which sends an $R$-module $M$ to the symmetric algebra $\Sym_R(M)$ over $R$ (which is constructed as a tensor algebra, in the same way as in the case of rings).

\begin{lemma}\label{lemma: symmetric algebra of an ideal}
 Let $I$ be an ideal of $R$. Then the canonical morphism
 \[
  \Psi_{R}: \ \Sym_R(I) \ \longrightarrow \ \lim\bigg(\prod_{g\in I} \Sym_R(I)[(g)^{-1}] \begin{array}{c}\longrightarrow\\[-7pt]\longrightarrow\end{array} \prod_{g,g'\in I} \Sym_R(I)[(gg')^{-1}] \bigg).
 \]
 is an isomorphism if and only if $I=R$.
\end{lemma}

\begin{proof}
 If $I\neq R$, then $1\notin I$ and the class $1=(\frac gg)_{g\in I}$ in the limit is not contained in the image of $\Psi_R$. If $I=R$, then $\1=\bigvee\{u_h\mid h\in R\}$ in $\Lambda_R$ and thus $\Psi_R$ is an isomorphism by \autoref{lemma: covering condition for the Stone dual}.
\end{proof}

%%%%%%%%%%%%%%%%%%%%%%%%%%%%%%%%%%%%%%%%%%%%%%%%%%%%%%%%%%%%%%%%%%%%%%%%%%%%%%%%%%%%%%%%%%%%%%%%%%%%%%%%%%%%%%%%%%%%%%%%%%%%%%%%%%%%%%%%%%%%%%%%%%%%%%%%%%%%%%%%%%%%%
\subsection{The inclusion of the prime spectrum into the Stone dual}

We begin with the verification of the inclusion $\PSpec R\subset\Lambda_R^\top$ and consider a point $\fp$ of $\PSpec R$, which is a prime ideal of $R$. We define
\[\textstyle
 \cP \ = \ \gen{\bigvee\chi_R(\fp)} \ = \ \{w\in\Lambda\mid w\leq\bigvee\chi_R(\fp)\}
\]
as the principal ideal of $\Lambda_R$ that is generated by the join $\bigvee\chi_R(\fp)$ of the image of $\fp$ in $\Lambda_R$. The following claim implies that $\fp\in\Lambda_R^\top$, as desired.

\begin{claim}
 The principal ideal $\cP$ is a prime ideal of $\Lambda_R$ and $\fp=\chi_R^{-1}(\cP)$.
\end{claim}

\begin{proof}
 By \autoref{lemma: localizations at primes}, we can assume that $\fp$ is the unique maximal $m$-ideal of $R$ and that $R^\times=R-\fp$. We claim that $\cS=\Lambda_R-\cP=\{\1\}$. The canonical morphism
 \[
  \Sym_{R}(\fp) \ \longrightarrow \ \lim\bigg(\prod_{g\in\fp} \Sym_R(\fp)[(g)^{-1}] \begin{array}{c}\longrightarrow\\[-7pt]\longrightarrow\end{array} \prod_{g,g'\in\fp} \Sym_R(\fp)[(gg')^{-1}]\bigg)
 \]
 is not an isomorphism by \autoref{lemma: symmetric algebra of an ideal}, since $\fp$ is a proper ideal of $R$. Thus $\bigvee\chi_R(\fp)\neq\1$ by \autoref{lemma: covering condition for the Stone dual}, and, in particular, $\1\in\cS$.

 Since $R[h^{-1}]=R$ for $h\in R^\times$, we have $u_h=\1$ for $h\in R$. Since $\fp=R-R^\times$, we have $w\leq\bigvee\chi_R(\fp)$ for all $w\neq\1$. Thus $\cS=\{\1\}$, which is closed under finite meets and shows that $\cP$ is a prime ideal. It shows further that $\chi_R^{-1}(\cP)=R-R^\times=\fp$, which completes the proof.
\end{proof}

%%%%%%%%%%%%%%%%%%%%%%%%%%%%%%%%%%%%%%%%%%%%%%%%%%%%%%%%%%%%%%%%%%%%%%%%%%%%%%%%%%%%%%%%%%%%%%%%%%%%%%%%%%%%%%%%%%%%%%%%%%%%%%%%%%%%%%%%%%%%%%%%%%%%%%%%%%%%%%%%%%%%%
\subsection{The inclusion of the Stone dual into the prime spectrum}
\label{subsetion: the inclusion of the Stone dual into the prime spectrum}

We turn to the proof of $\Lambda^\top_R\subset\PSpec R$. Let $\cP$ be a principal prime ideal of $\Lambda_R^\top$ and $P=\chi_R^{-1}(\cP)$ the corresponding prime $m$-ideal of $R$. We aim to show the following.

\begin{claim}
 The prime $m$-ideal $P$ is in $\PSpec R$.
\end{claim}

\begin{proof}
 By \autoref{lemma: localizations at primes}, we can assume that $P$ is the unique maximal $m$-ideal of $R$ and $R^\times=R-P$. Let $\fp$ be the ideal generated by $P$, which is the additive closure of $P$ and equal to either $P$ or $R$. If $P=\fp$, then $P$ is a prime ideal, as claimed. Thus it suffices to lead the assumption $\fp=R$ to a contradiction.

 If this was the case, then $\PSpec R$ is covered by the family $\{U_h^P\mid h\in P\}$ of open subsets by \autoref{lemma: covering condition for PSpec R}, and the same is true if we take inverse images under a morphism $f^\ast:\PSpec T\to\PSpec R$ induced by a semiring morphism $f:R\to T$. Thus $\PSpec T$ is equal to the union of the open subsets $f^\ast(U^P_h)$ glued along their pairwise intersections $f^\ast(U^P_{gh})$. In terms of their coordinate semirings, this means that the canonical semiring homomorphism
 \[
  T \ \longrightarrow \ \lim\bigg(\prod_{h\in P} T[h^{-1}] \begin{array}{c}\longrightarrow\\[-7pt]\longrightarrow\end{array} \prod_{g,h\in P} T[(gh)^{-1}] \bigg)
 \]
 is an isomorphism. By \autoref{lemma: covering condition for the Stone dual}, we conclude that $\bigvee\{u_h\mid h\in P\}=\1$, which is the desired contradiction since $\1\notin\cP$.
\end{proof}

This concludes the proof of \autoref{thmA}.

%%%%%%%%%%%%%%%%%%%%%%%%%%%%%%%%%%%%%%%%%%%%%%%%%%%%%%%%%%%%%%%%%%%%%%%%%%%%%%%%%%%%%%%%%%%%%%%%%%%%%%%%%%%%%%%%%%%%%%%%%%%%%%%%%%%%%%%%%%%%%%%%%%%%%%%%%%%%%%%%%%%%%
%%%%%%%%%%%%%%%%%%%%%%%%%%%%%%%%%%%%%%%%%%%%%%%%%%%%%%%%%%%%%%%%%%%%%%%%%%%%%%%%%%%%%%%%%%%%%%%%%%%%%%%%%%%%%%%%%%%%%%%%%%%%%%%%%%%%%%%%%%%%%%%%%%%%%%%%%%%%%%%%%%%%%

\section{Proof of \autoref{thmB}}
\label{ProofB}

In this section we explain the proof for \autoref{thmB}. In the favour of a more approachable text, we restrict ourselves to proof outlines of the more technical parts. We refer to \cite{toolkit} for full details.

%%%%%%%%%%%%%%%%%%%%%%%%%%%%%%%%%%%%%%%%%%%%%%%%%%%%%%%%%%%%%%%%%%%%%%%%%%%%%%%%%%%%%%%%%%%%%%%%%%%%%%%%%%%%%%%%%%%%%%%%%%%%%%%%%%%%%%%%%%%%%%%%%%%%%%%%%%%%%%%%%%%%%
\subsection{Morphisms and refinements of \texorpdfstring{$\fs$}{s}-presentations}
The main tools in the proof of \autoref{thmB} are morphisms and refinements of $\fs$-presentations, as introduced in the following.

Let $\cU:\cI\to\Aff$ and $\cV:\cJ\to\Aff$ be $\fs$-presentations. A \emph{morphism $\cV\to\cU$} is a pair of a functor $\underline\Phi:\cJ\to \cI$ between the respective index categories together with a natural transformation $\Phi:\cV\to\cU\circ\underline\Phi$. We typically refer to such a morphism by $\Phi:\cV\to\cU$, omitting $\underline\Phi$ from the notation. Note that $\Phi$ induces a morphism $\colim\Phi:\colim_{\Sh_\fs}\cV\to\colim_{\Sh_\fs}\cU$ between the respective colimits, which are $\fs$-schemes.

A \emph{refinement of an $\fs$-presentation $\cU$} is a certain type of morphisms $\Omega:\cW\to\cU$ of $\fs$-presentations, whose key property is that the induced morphism $\colim\Omega:\colim_{\Sh_\fs}\cW\to\colim_{\Sh_\fs}\cU$ is an isomorphism. We omit the exact definition from \cite{toolkit}, and rather use the validity of this key property as the working definition of a refinement in this text.

The proof of \autoref{thmB} rests on the following results on affine presentations and their refinements. Detailed proofs appear in \cite{toolkit}.

\begin{thm}\label{thm: refinements}
 Let $\cU$ and $\cV$ be $\fs$-presentations with respective colimits $X=\colim_{\Sh_\fs}\cU$ and $Y=\colim_{\Sh_\fs}\cV$ in $\Sh_\fs$. Let $\varphi:Y\to X$ be a morphism of sheaves on $\fs$. Then there is a refinement $\Omega:\cW\to\cV$ and a morphism $\Phi:\cW\to\cU$ such that $\varphi=\colim\Phi$, i.e., the diagram
 \[
  \begin{tikzcd}[column sep=80, row sep=15]
   Y \ar[r,"\varphi"] & X \\
   \colim_{\Sh_\fs}\cV \ar[u,"\simeq"'] \\
   \colim_{\Sh_\fs}\cW \ar[u,"\simeq"',"\colim\Omega"] \ar[r,"\colim\Phi"] & \colim_{\Sh_\fs}\cU \ar[uu,"\simeq"']
  \end{tikzcd}
 \]
 commutes.
\end{thm}

\begin{proof}[Proof idea]
 In a first attempt, we define $\overline\cW$ as the fibre product of $\cV$ and $\cU$ over $X$, i.e., $\overline\cW$ is the diagram that consists of all pairwise fibre products $U_i\times_X V_j$ of $U_i=\cU(i)$ and $V_j=\cV(j)$ over $X$, where $V_j\to X$ is the composition of the monomorphism $V_j\to Y$ with $\varphi:Y\to X$. It comes with canonical morphisms $\overline\cW\to\cV$ and $\overline\cW\to\cU$. Since $\colim_{\Sh_\fs}\cU\to X$ is an isomorphism, it follows that $\colim_{\Sh_\fs}\overline\cW\to\colim_{\Sh_\fs}\cV$ is an isomorphism. If $\overline\cW$ is a diagram in $\Aff$, then $\cW=\overline\cW$ together with the canonical morphisms $\Omega:\cW\to\cV$ and $\Phi:\cW\to\cU$ verifies the assertion of the theorem.

 The fibre products $U_i\times_X V_j$ are, however, in general not affine. This requires us to replace the $U_i\times_XV_j$ by $\fs$-presentations and the canonical projections by suitable morphisms, which piece together to an $\fs$-presentation $\cW$. In addition, $\cW$ comes with a morphism $\cW\to\overline\cW$ that induces an isomorphism on the respective colimits in $\Sh_\fs$. The compositions $\Omega:\cW\to\overline\cW\to\cV$ and $\Phi:\cW\to\overline\cW\to\cU$ satisfy the desired properties.
\end{proof}

\autoref{thm: refinements} has two immediate consequences.

\begin{cor}\label{cor: morphisms of s-schemes are affinely presented}
 Every morphism of $\fs$-schemes is the colimit of a morphism of $\fs$-presentations. \qed
\end{cor}

Let $X$ be an $\fs$-scheme. An \emph{$\fs$-presentation of $X$} is an affine $\fs$-presentation $\cU$ together with an isomorphism $\alpha:\colim_{\Sh_\fs}\cU\to X$.

\begin{cor}\label{cor: common refinements}
 Any two $\fs$-presentations $\colim_{\Sh_\fs}\cU\simeq X$ and $\colim_{\Sh_\fs}\cV\simeq X$ of an $\fs$-scheme $X$ have a common refinement, i.e., there exists an $\fs$-presentation $\cW$ together with refinements $\Omega:\cW\to \cU$ and $\Theta:\cW\to\cV$.
\end{cor}

\begin{proof}
 This follows by applying \autoref{thm: refinements} to the identity morphism $\varphi=\id_X$ of $Y=X$.
\end{proof}

%%%%%%%%%%%%%%%%%%%%%%%%%%%%%%%%%%%%%%%%%%%%%%%%%%%%%%%%%%%%%%%%%%%%%%%%%%%%%%%%%%%%%%%%%%%%%%%%%%%%%%%%%%%%%%%%%%%%%%%%%%%%%%%%%%%%%%%%%%%%%%%%%%%%%%%%%%%%%%%%%%%%%
\subsection{The image of \texorpdfstring{$\fs$}{s}-presentations}

Let $\tau:\Aff\to\Aff_\N$ be the equivalence with $\tau(\Spec R)=\PSpec R$. In this section, we study the image $\tau\cU=\tau\circ\cU$ of an $\fs$-presentation $\cU:\cI\to\Aff$, which is a diagram in $\Aff_\N$.

A \emph{principal open immersion of semiring schemes} is a morphism $\varphi:Y\to X$ of affine semiring schemes that is induced by the localization $R\to R[h^{-1}]$ of $R=\cO_X(X)$ at an element $h\in R$. By \cite[Lemma 3.3 and Thm.\ 3.18]{GKMX}, a principal open immersion is an open immersion (cf.\ \autoref{subsubsection: semiring schemes as semiringed spaces}). The following result establishes the last claim of \autoref{thmB}.

\begin{lemma}\label{lemma: image of s-presentations}
 Let $\cU:\cI\to\Aff$ be a diagram. Then the following are equivalent:
 \begin{enumerate}
  \item $\cU$ is an $\fs$-presentation;
  \item $\tau\cU:\cI\to\Aff_\N$ is a diagram of principal open immersions that is closed under fibre products such that the canonical inclusion $\tau\cU(i)\to\colim_{\Sp_\N}\tau\cU$ is an open immersion for every $i\in\cI$.
 \end{enumerate}
\end{lemma}

\begin{proof}[Proof idea]
 By \autoref{thmA}, principal open immersions $U_h\to \Spec R$ of affine $\fs$-schemes correspond to principal open immersions $U_h^P\to \PSpec R$ of prime spectra. Since $\Aff_\N$ is equivalent to $\Aff$, the diagram $\cU$ is closed under fibre products if and only if $\tau\cU$ is.

 The monodromy freeness of $\cU$ can be characterized intrinsically in terms of the properties of $\cU$ as a functor into $\Aff$; cf.\ \cite[section 3.2]{BLMV}. Since $\tau:\Aff\to\Aff_\N$ is an equivalence, $\cU$ satisfies this intrinsic property if and only if $\tau\circ\cU$ satisfies the corresponding property with respect to $\Aff_\N$. Similarly as for $\cU$, this intrinsic property for $\tau\circ\cU$ is equivalent to the property that the canonical inclusions $\iota_i:\tau \cU(i)\to \colim_{\Sp_\N}\tau\cU$ are monomorphisms for all $i\in\cI$. Thus $\cU$ is an $\fs$-presentation if and only if $\tau\circ\cU$ is a diagram of principal open immersions and $\iota_i$ is a monomorphism for all $i\in\cI$.

 We are left with showing that $\iota_i$ is an open topological embedding. To start with, monomorphisms in $\Sp_\N$ are injective maps. By the defining property of the colimit of topological spaces, a subset $W$ of $X=\colim_{\Sp_\N}\tau\cU$ is open if and only if $\iota_i^{-1}(W)$ is open in $\tau\cU(i)$ for all $i\in\cI$. Since principal open immersions are open topological embeddings, it suffices to verify this condition for a covering of $W$ by subsets of the form $\iota_i(\tau\cU(i))\cap W$. In particular, if $W$ is contained in $\iota_i(\tau\cU(i))$ for a single $i\in\cI$, then $W$ is open in $X$ if and only if $\iota_i^{-1}(W)$ is open in $\tau\cU(i)$. This shows that $\iota_i$ is indeed an open topological embedding.

 That the structure sheaf of $U_i$ is the restriction of the structure sheaf of $X$ follows in a similar way from the fact that this is the case for the principal open immersions in $\tau\cU$.
\end{proof}

\begin{lemma}\label{lemma: image of refinements under tau}
 Let $\Omega:\cW\to\cU$ be a refinement of $\fs$-presentations. Then $\colim_{\Sp_\N}\tau\Omega:\colim_{\Sp_\N}\tau\cW\to\colim_{\Sp_\N}\tau\cU$ is an isomorphism.
\end{lemma}

\begin{proof}[Remarks on the proof]
 While this statement is true for our working definition of refinements as morphisms of $\fs$-presentations whose colimit in $\Sh_\fs$ is an isomorphism, its proof requires the more restrictive definition that we omit from this text. This definition is based on explicit conditions that allow us to deduce (with a moderate effort) that both $\colim_{\Sh_\fs}\Omega$ as well as $\colim_{\Sp_\N}\tau\Omega$ are isomorphisms.
\end{proof}

%%%%%%%%%%%%%%%%%%%%%%%%%%%%%%%%%%%%%%%%%%%%%%%%%%%%%%%%%%%%%%%%%%%%%%%%%%%%%%%%%%%%%%%%%%%%%%%%%%%%%%%%%%%%%%%%%%%%%%%%%%%%%%%%%%%%%%%%%%%%%%%%%%%%%%%%%%%%%%%%%%%%%
\subsection{Uniqueness of \texorpdfstring{$\hat\tau$}{t}}

We turn to the proof of \autoref{thmB}, which is divided into the upcoming sections. We begin with the uniqueness of $\hat\tau:\Sch_\fs\to\Sch_\N$, which follows, in essence, from the property that it it preserves the colimits of affine presentations.

On the level of objects, this follows since every $\fs$-scheme $X$ is, by definition, the colimit of an affine $\fs$-presentation $\cU$, and thus $\hat\tau X\simeq\colim_{\Sp_\N}\tau\cU$ determines $\hat\tau X$ uniquely (up to isomorphism). On the level of morphisms $\varphi:Y\to X$, this follows since $\varphi$ is the colimit of a morphism $\Phi:\cV\to\cU$ of $\fs$-presentations, and thus $\hat\tau\varphi:\hat\tau Y\to\hat\tau X$ is the colimit of $\tau\Phi:\tau\cV\to\tau\cU$. These requirements determine $\hat\tau$ uniquely, up to a uniquely determined natural equivalence.

%%%%%%%%%%%%%%%%%%%%%%%%%%%%%%%%%%%%%%%%%%%%%%%%%%%%%%%%%%%%%%%%%%%%%%%%%%%%%%%%%%%%%%%%%%%%%%%%%%%%%%%%%%%%%%%%%%%%%%%%%%%%%%%%%%%%%%%%%%%%%%%%%%%%%%%%%%%%%%%%%%%%%
\subsection{Existence of \texorpdfstring{$\hat\tau$}{t}}

We turn to the proof that there is a functor $\hat\tau:\Sch_\fs\to\Sch_\N$ that extends $\tau$ and preserves the colimits of $\fs$-presentations. For this, we have to show that the formulas $\hat\tau X\simeq \colim_{\Sp_\N}\tau\cU$ and $\hat\tau\varphi\simeq \colim_{\Sp_\N}\tau\Phi$ are independent of the choice of an $\fs$-presentation $\cU$ of $X$ and of a choice of a morphism $\Phi$ of $\fs$-presentations with $\varphi=\colim_{\Sh_\fs}\Phi$.

Assume that $X\simeq\colim_{\Sh_\fs}\cV$ for a second $\fs$-presentation $\cV$. We want to show that $\colim_{\Sp_\N}\tau\cU\simeq\colim_{\Sp_\N}\tau\cV$. By \autoref{cor: common refinements}, $\cU$ and $\cV$ have a common refinement $\cW$. By \autoref{lemma: image of refinements under tau}, the image of the refinements $\cW\to\cU$ and $\cW\to\cV$ induce isomorphisms
\[
 \colim_{\Sp_\N}\tau\cU \ \stackrel\simeq\longleftarrow \ \colim_{\Sp_\N}\tau\cV \ \stackrel\simeq\longrightarrow \ \colim_{\Sp_\N}\tau\cV,
\]
which establish the desired claim.

Consider a morphism $\varphi:Y\to X$ that is the colimit of two morphisms $\Phi_i:\cV_i\to\cU_i$ of $\fs$-presentations for $X$ and $Y$ (for $i=1,2$), i.e., the diagram
\[
 \begin{tikzcd}[column sep=80, row sep=15]
  \colim_{\Sh_\fs}\cV_1 \ar[d,"\simeq"'] \ar[r,"\colim_{\Sh_\fs}\Phi_1"] & \colim_{\Sh_\fs}\cU_1\ar[d,"\simeq"'] \\
  Y \ar[r,"\varphi"] & X \\
  \colim_{\Sh_\fs}\cV_2 \ar[u,"\simeq"] \ar[r,"\colim_{\Sh_\fs}\Phi_2"] & \colim_{\Sh_\fs}\cU_2\ar[u,"\simeq"]
 \end{tikzcd}
\]
commutes. We aim to show that applying $\hat\tau$ yields a commutative diagram in $\Sp_\N$. This follows from choosing first common refinements $\cV$ of $\cV_1$ and $\cV_2$ and $\cU$ of $\cU_1$ and $\cU_2$, and then choosing a further refinement $\cW\to\cV$ that allows a morphism $\Phi:\cW\to\cU$ with $\varphi=\colim_{\Sh_\fs}\Phi$. This yields a commutative diagram
\[
 \begin{tikzcd}[column sep=80, row sep=15]
  \colim_{\Sp_\N}\tau\cV_1 \ar[r,"\colim_{\Sp_\N}\tau\Phi_1"] & \colim_{\Sh_\fs}\tau\cU_1 \\
  \colim_{\Sp_\N}\tau\cW \ar[r,"\colim_{\Sp_\N}\tau\Phi"] \ar[u,"\simeq"] \ar[d,"\simeq"] & \colim_{\Sp_\N}\tau\cU \ar[u,"\simeq"] \ar[d,"\simeq"] \\
  \colim_{\Sp_\N}\tau\cV_2 \ar[r,"\colim_{\Sp_\N}\tau\Phi_2"] & \colim_{\Sp_\N}\tau\cU_2,
 \end{tikzcd}
\]
which establishes the desired claim.

%%%%%%%%%%%%%%%%%%%%%%%%%%%%%%%%%%%%%%%%%%%%%%%%%%%%%%%%%%%%%%%%%%%%%%%%%%%%%%%%%%%%%%%%%%%%%%%%%%%%%%%%%%%%%%%%%%%%%%%%%%%%%%%%%%%%%%%%%%%%%%%%%%%%%%%%%%%%%%%%%%%%%
\subsection{Surjectivity of \texorpdfstring{$\hat\tau$}{t}}

Next we show that the functor $\hat\tau:\Sch_\fs\to \Sch_\N$ is essentially surjective. Let $Z$ be a semiring scheme. Then $Z$ has an affine open cover $Z=\bigcup W_i$ by affine semiring schemes $W_i$, and each intersection $W_i\cap W_j=W_i\times_XW_j$ is covered by principal open subschemes $W_{ijk}$. The collection of all $W_i$ and $W_{ijk}$, together with the principal open immersions $W_{ijk}\to W_i$ form a diagram $\cW'$, whose colimit in $\Sp_\N$ is $X$. Adding repeatedly all fibre products to $\cW'$ yields a diagram $\cW$ of principal open immersions that is closed under fibre products, whose colimit in $\Sp_\N$ is equal to that of $\overline\cW$ since $\overline\cW$ is a cofinal subdiagram of $\cW$.

Let $\cU=\tau^{-1}\circ\cW$, where $\tau^{-1}$ is an inverse to $\tau$, i.e., $\cW\simeq\tau\cU$. By \autoref{lemma: image of s-presentations}, $\cU$ is an $\fs$-presentation. Thus
\[
 Z \ \simeq \ \colim_{\Sp_\N}\cW \ \simeq \ \colim_{\Sp_\N}\tau\cU \ \simeq \ \hat\tau(\colim_{\Sh_\fs}\cU),
\]
which shows that $\hat\tau$ is essentially surjective.

%%%%%%%%%%%%%%%%%%%%%%%%%%%%%%%%%%%%%%%%%%%%%%%%%%%%%%%%%%%%%%%%%%%%%%%%%%%%%%%%%%%%%%%%%%%%%%%%%%%%%%%%%%%%%%%%%%%%%%%%%%%%%%%%%%%%%%%%%%%%%%%%%%%%%%%%%%%%%%%%%%%%%
\subsection{Fullness of \texorpdfstring{$\hat\tau$}{t}}

Next we show that $\hat\tau:\Sch_\fs\to \Sch_\N$ is full. Let $X$ and $Y$ be $\fs$-schemes with respective $\fs$-presentations $X\simeq\colim_{\Sh_\fs}\cU$ and $Y\simeq\colim_{\Sh_\fs}\cV$. Let $\psi:\hat\tau Y\to\hat\tau X$ be a morphism of semiring schemes. We aim to show that $\psi=\hat\tau\varphi$ for a morphism $\varphi:Y\to X$.

By the defining property of $\hat\tau$, $X=\colim_{\Sp_\N}\tau\cU$ and $Y=\colim_{\Sp_\N}\tau\cV$ and by \autoref{lemma: image of s-presentations}, the canonical inclusions $\tau\cU(i)\to X$ and $\tau\cV(j)\to Y$ are open immersions. Then the inverse image $\psi^{-1}(W)$ of an open subset $W\subset X$ is open in $Y$ and can thus be covered by affine open subschemes of $Y$.

If we choose such affine open coverings for the inverse images of $\iota_i(\tau\cU(i))$ under $\psi$ in a compatible way for all $i$, then this yields a refinement $\cW\to\cV$ that allows for a morphism $\Phi:\cW\to\cU$ whose colimit $\varphi:Y\to X$ satisfies $\hat\tau\varphi=\colim_{\Sp_\N}\tau\Phi=\varphi$, as desired. We omit the details of this last step, which are analogous to the proof of \autoref{thm: refinements}.

%%%%%%%%%%%%%%%%%%%%%%%%%%%%%%%%%%%%%%%%%%%%%%%%%%%%%%%%%%%%%%%%%%%%%%%%%%%%%%%%%%%%%%%%%%%%%%%%%%%%%%%%%%%%%%%%%%%%%%%%%%%%%%%%%%%%%%%%%%%%%%%%%%%%%%%%%%%%%%%%%%%%%
\subsection{Faithfulness of \texorpdfstring{$\hat\tau$}{t}}

We are left to show that $\hat\tau:\Sch_\fs\to\Sch_\N$ is faithful. Consider two morphisms $\varphi_1,\varphi_2:Y\to X$ of $\fs$-schemes with $\hat\tau\varphi_1=\hat\tau\varphi_2$. We choose $\fs$-presentations $X\simeq \colim_{\Sh_\fs}\cU$ and $Y\simeq \colim_{\Sh_\fs}\cV$. By \autoref{thm: refinements}, there are refinements $\cW_1\to \cV$ and $\cW_2\to\cV$ and morphisms $\Phi_1':\cW_1\to\cU$ and $\Phi_2':\cW_2\to\cU$ whose colimits are $\varphi_1$ and $\varphi_2$, respectively.

Passing to a common refinement $\cW$ of $\cW_1$ and $\cW_2$ yields morphisms $\Phi_1:\cW\to\cW_1\to\cU$ and $\Phi_2:\cW\to\cW_2\to\cU$ whose respective colimits in $\Sh_\fs$ are $\varphi_1$ and $\varphi_2$. By assumption, we have
\[
 \colim_{\Sp_\N}\tau\Phi_1 \ = \ \hat\tau\varphi_1 \ = \ \hat\tau\varphi_2 \ = \ \colim_{\Sp_\N} \tau\Phi_2.
\]
Since the canonical inclusions $\tau\Phi_i:\tau\cU(i)\to \tau Y$ are monomorphisms, we have $\tau\Phi_1(i)=\tau\Phi_2(i)$ for all $i\in\cI$, and thus $\Phi_1(i)=\Phi_2(i)$ since $\tau$ is an equivalence. This shows that $\Phi_1=\Phi_2$ and thus $\varphi_1=\varphi_2$, as desired.

This concludes the proof of \autoref{thmB}.

%%%%%%%%%%%%%%%%%%%%%%%%%%%%%%%%%%%%%%%%%%%%%%%%%%%%%%%%%%%%%%%%%%%%%%%%%%%%%%%%%%%%%%%%%%%%%%%%%%%%%%%%%%%%%%%%%%%%%%%%%%%%%%%%%%%%%%%%%%%%%%%%%%%%%%%%%%%%%%%%%%%%%
%%%%%%%%%%%%%%%%%%%%%%%%%%%%%%%%%%%%%%%%%%%%%%%%%%%%%%%%%%%%%%%%%%%%%%%%%%%%%%%%%%%%%%%%%%%%%%%%%%%%%%%%%%%%%%%%%%%%%%%%%%%%%%%%%%%%%%%%%%%%%%%%%%%%%%%%%%%%%%%%%%%%%
\begin{small}
 \bibliographystyle{plain}
 \bibliography{f1}

\begin{thebibliography}{1}

\bibitem{BLMV}
Sourayan Banerjee, Oliver Lorscheid, Alejandro Martínez~Méndez, and Alejandro
  Vargas.
\newblock Toolkit for the algebraic geometer.
\newblock Lecture notes of a minicourse from the 2025 Barcelona workshop on the
  Geometry of Semirings. Online available at \arxiv{2603.26500}, 2026.

\bibitem{Chu-Lorscheid-Santhanam12}
Chenghao Chu, Oliver Lorscheid, and Rekha Santhanam.
\newblock Sheaves and {$K$}-theory for {$\mathbb F_1$}-schemes.
\newblock {\em Adv. Math.}, 229(4):2239--2286, 2012.

\bibitem{GKMX}
Roberto Gualdi, Arne Kuhrs, Mayo Mayo~Garcia, and Xavier Xarles.
\newblock Affine scheme theory for commutative semirings.
\newblock Lecture notes of a minicourse from the 2025 Barcelona workshop on the
  Geometry of Semirings. Online available at \arxiv{2601.14136}, 2026.

\bibitem{toolkit}
Oliver Lorscheid and Alejandro Martínez~Méndez.
\newblock Toolkit for the algebraic geometer.
\newblock Paper with two parts, in preparation, 2026.

\bibitem{Marty07}
Florian Marty.
\newblock Relative {Z}ariski open objects.
\newblock {\em J. K-Theory}, 10(1):9--39, 2012.

\bibitem{Toen-Vaquie09}
Bertrand To{\"e}n and Michel Vaqui{\'e}.
\newblock Au-dessous de {${\rm Spec}\,\mathbb Z$}.
\newblock {\em J. K-Theory}, 3(3):437--500, 2009.

\end{thebibliography}
\end{small}
%%%%%%%%%%%%%%%%%%%%%%%%%%%%%%%%%%%%%%%%%%%%%%%%%%%%%%%%%%%%%%%%%%%%%%%%%%%%%%%%%%%%%%%%%%%%%%%%%%%%%%%%%%%%%%%%%%%%%%%%%%%%%%%%%%%%%%%%%%%%%%%%%%%%%%%%%%%%%%%%%%%%%
%%%%%%%%%%%%%%%%%%%%%%%%%%%%%%%%%%%%%%%%%%%%%%%%%%%%%%%%%%%%%%%%%%%%%%%%%%%%%%%%%%%%%%%%%%%%%%%%%%%%%%%%%%%%%%%%%%%%%%%%%%%%%%%%%%%%%%%%%%%%%%%%%%%%%%%%%%%%%%%%%%%%%

\end{document}